\documentclass[12pt,reqno,a4paper]{amsart}
\usepackage[ansinew]{inputenc}
\usepackage{amsfonts}
\usepackage{latexsym}
\usepackage{mathrsfs}
\usepackage{amsmath}
\usepackage{amssymb}

\usepackage{eepic}
\usepackage{color}
\usepackage[margin=3.6cm]{geometry}
\usepackage{hyperref}
\hypersetup{
    colorlinks=true,
    linkcolor=blue,
    citecolor=red,}

\newcommand{\C}{\mathbb{C}}

\newcommand{\Z}{\mathbb{Z}}

\newcommand{\U}{\mathcal{U}}

\newcommand{\uld}{\overline{\operatorname{logdens}}}

\newcommand{\ud}{\overline{\operatorname{dens}}}

\newtheorem{defin}{Definition}[section]

\newtheorem{theorem}[defin]{Theorem}
\newtheorem{exa}[defin]{Example}
\newenvironment{example}{\begin{exa}\rm}{\end{exa}}\newtheorem{lemma}[defin]{Lemma}
\newtheorem{corollary}[defin]{Corollary}
\newenvironment{remark}
{\par\vspace{0.5mm}\noindent{\bf Remark.}}{\par\vspace{0.5mm}}

\numberwithin{equation}{section}

\makeatletter
\renewcommand{\ps@myheadings}{%
\renewcommand{\@evenhead}%
{{\rm\thepage}\hfil{\sc Heittokangas, Pulkkinen, Yu, Zemirni}\hfil}%
\renewcommand{\@oddhead}%
{\hfil{{\sc Standard solutions of complex differential equations}\hfil{\rm\thepage}}}%
\renewcommand{\@evenfoot}{}%
\renewcommand{\@oddfoot}{}%
}\makeatother 
 
\title{Standard solutions of complex linear differential equations}

\begin{document}

\author{J.~Heittokangas*, S.~Pulkkinen, H.~Yu and M.~A.~Zemirni}

\maketitle

\begin{abstract}
A meromorphic solution of a complex linear differential equation (with meromorphic coefficients) for which the value zero is the only possible finite deficient/deviated value is called a standard solution. Conditions for the existence and the number of standard solutions are discussed for various types of deficient and deviated values. 


\end{abstract}

\renewcommand{\thefootnote}{}
\footnotetext[1]{*Corresponding author.}
\footnotetext[2]{The second author was supported by the Vilho, Yrj\"o and Kalle V\"ais\"al\"a Foundation of the Finnish Academy of Science and Letters, and by the Oskar \"Oflunds Stiftelse sr.}

%
%

\section{Background}\label{Back-sec}

\thispagestyle{empty}

We focus on deficient and deviated values of solutions of linear differential equations 
    \begin{equation}\label{lden}
        f^{(n)}+A_{n-1}(z)f^{(n-1)}+\cdots+A_1(z)f'+A_0(z)f=0 
    \end{equation}
with entire or meromorphic coefficients $A_0,\ldots, A_{n-1}$. The solutions are known to be entire in the case of entire coefficients, while the existence of meromorphic solutions is not guaranteed if the coefficients are meromorphic. For example, the equation
	$$
	f''+2z^{-1}f'-z^{-4}f=0
	$$
with rational coefficients has a non-meromorphic solution $f(z)=\exp\left(z^{-1}\right)$.

For a meromorphic function $f$ in $\C$, and for $a\in\widehat{\C}$, we define the quantities
    \begin{equation*}
    \delta_N(a,f):=\liminf_{r\to\infty}\frac{m(r,a,f)}{T(r,f)}, 
     \end{equation*}
     \begin{equation*}
    \delta_P(a,f):=\liminf_{r\to\infty}\frac{\mathscr{L}(r,a,f)}{T(r,f)},
    \end{equation*}
where $m(r,a,f)$ denotes the proximity function, $T(r,f)$ is the Nevanlinna
characteristic of $f$, and  
	\begin{equation*}
	\mathscr{L}\left(r,a,f\right):=\left\{\begin{array}{ll}
	\max_{|z|=r}\log^+ \frac{1}{|f(z)-a|}, &\ a\in \mathbb{C} ,\\[7pt]
	\max_{|z|=r}\log^+ {|f(z)|}, &\ a=\infty,
	\end{array}\right.
	\end{equation*}
is the \emph{logarithmic maximum modulus} for the $a$-points of $f$. It is clear that
    $$
    0\leq\delta_N(a,f)\leq1\quad \text{and}\quad0\leq\delta_N(a,f)\leq\delta_P(a,f)\leq\infty.
    $$
The equality $\delta_P(a,f)=\infty$ is possible, for example, when $f(z)=\exp\big(e^z\big)$ and $a=\infty$. Indeed, in this case $\mathscr{L}(r,\infty,f)=e^r$, while (see \cite[p.~7]{Hayman})
	$$
	T(r,f)\sim \frac{e^r}{\sqrt{2\pi^3 r}}.
	$$

The quantity $\delta_N(a,f)$, introduced by  Nevanlinna, can be written alternatively as
	$$
	\delta_N(a,f)=1-\limsup_{r\to\infty}\frac{N(r,a,f)}{T(r,f)},
	$$
where $N(r,a,f)$ is the integrated counting function of the $a$-points of $f$. If $\delta_N(a,f)>0$, then $a$ is called a \emph{deficient value} for $f$ because $f$ attains the value $a$ less often than the growth of the characteristic function $T(r,f)$ would allow. 
Meanwhile, if $\delta_P(a,f)>0$, then $a$ is called a \emph{Petrenko deviated value} for $f$, and, following \cite{BP, Petrenko}, the quantity  $\delta_P(a,f)$ is called the \emph{magnitude of the deviation} of $f$ from $a$. 

It is known that the set of deficient values is at most countable \cite{Hayman}, while the set of P-deviated values is of zero capacity but could be uncountable for functions of infinite lower order $\mu$ \cite{BP}. Given a meromorphic function $f$, the sum of all N-deficient values for $f$ is $\leq 2$ \cite{Hayman}, while the sum of all P-deviated values for $f$ is $\leq K(\mu+1)$ for some constant $K>0$  \cite{BP}.

A solution $f$ of \eqref{lden} satisfying  $\delta_P(a,f)=0$ (resp.~$\delta_N(a,f)=0$) for every $a\in\mathbb{C}\setminus\{0\}$ is called a \emph{P-standard solution} (resp.~\emph{N-standard solution}). The notion ``standard'' is from Petrenko \cite{BP}, but the prefixes are added in order to identify the right quantity we are dealing with.  In particular, a P-standard solution is also an N-standard solution, but not necessarily conversely.

As for results on the equation \eqref{lden}, we begin with a well-known result that is originally due to H.~Wittich \cite{Wittich}. This result is often considered to be one of the corner stones of the oscillation theory of complex differential equations.

\begin{theorem}[{\cite[Theorem~4.3]{Laine}}]\label{N-standard-thm}
Suppose that the coefficients $A_0,\ldots, A_{n-1}$ in \eqref{lden} are meromorphic, and that $f$ is an admissible meromorphic solution of \eqref{lden} in the sense that 
    \begin{equation}\label{Wittich-assumption}
    T(r,A_j)=o(T(r,f)),\quad r\not\in E,\, j=0,\ldots,n-1,
    \end{equation}
where $E\subset[0,\infty)$ is a set of finite linear measure.   
Then $0$ is the only possible finite deficient value for $f$.
\end{theorem}

In our terminology, the solution $f$ in Theorem~\ref{N-standard-thm} is an N-standard solution. Note that if the coefficients $A_0,\ldots, A_{n-1}$ are rational and if $f$ is a transcendental meromorphic solution of \eqref{lden}, then \eqref{Wittich-assumption} is clearly valid.

Exceptional sets of finite linear/logarithmic measure are very typical in Nevanlinna theory and its applications. However, it is apparent from the proof of Theorem~\ref{N-standard-thm} in \cite{Laine} that the set $E$ in \eqref{Wittich-assumption} could be much larger. For example, we could equally well assume that $\ud (E)<1$, where 
	$$
	\ud(E):=\limsup_{r\to\infty}\frac{\int_{E\cap [0,r]}dt}{r}
	$$
is the upper linear density of $E$. It is clear that $0\leq \ud (F)\leq 1$ for every set $F\subset [0,\infty)$, and that $\ud (F)=0$ whenever $F$ has finite linear measure.  

Theorem~\ref{N-standard-thm} brings us to the question of how typical it is for \eqref{lden} to possess an N-standard solution? A partial answer lies in the following result.

\begin{theorem}[{\cite[Theorem~2.3]{HYZ}}]\label{asymptotic-cor}
Let the coefficients $A_{0},\ldots,A_{n-1}$ in \eqref{lden} be entire functions such that at least one of them is transcendental. Suppose that $p\in\{0,\ldots,n-1\}$ is the smallest index such that
	\begin{equation}\label{2LM2}
	\limsup_{\substack{r\to \infty }}\sum_{j=p+1}^{n-1}\frac{\log^+ M(r,A_{j})}{\log^+ M(r,A_{p})}<1.
	\end{equation}
Then $A_p$ is transcendental, and every solution base of \eqref{lden} has $m\geq n-p$ solutions $f$ for which
 	\begin{equation}\label{result-th1.3}
 	\log T(r,f)\asymp\log M(r,A_p),\quad r\not\in E,
 	\end{equation}
where $E \subset [0,\infty)$ has finite linear measure. For each such solution $f$, the value $0$ is the only possible finite deficient value.
\end{theorem}

In our terminology, the condition \eqref{2LM2} induces $m\geq n-p$ N-standard solutions in every solution base of \eqref{lden}. It follows  that every solution base of \eqref{lden} with entire coefficients has at least one N-standard solution, with no conditions on the coefficients other than that they must be entire. 

We proceed to discuss known results on P-standard solutions of \eqref{lden}.

\begin{theorem}[{\cite[Theorem~2]{BP}}]\label{atleast-one-standard-thm}
Suppose that the coefficients $A_0,\ldots, A_{n-1}$ in \eqref{lden} are entire. Then every solution base of \eqref{lden} has at least one P-standard solution.
\end{theorem}

Theorem 3 in \cite{BP} shows that there are equations \eqref{lden} with entire coefficients and $n\geq 2$, which have $n-1$ linearly independent solutions that are not  P-standard.  This proves the sharpness of Theorem~\ref{atleast-one-standard-thm}. 

To discuss further results on P-standard solutions, we suppose that the coefficients of \eqref{lden} are entire, and recall that a \emph{characteristic function} of \eqref{lden}  is defined by
	\begin{equation}\label{characteristic-function}
	 T(r):=\frac{1}{2\pi}\int_0^{2\pi}\log \sqrt{1+\sum_{k=1}^n|f_k(re^{i\theta})|^2} \,d\theta,
	\end{equation} 
where $\{f_1,\ldots,f_n\}$ is a solution base for \eqref{lden}. The function $T(r)$ is essentially independent of the solution base used in defining it in the following sense.

\begin{lemma}
If $T_1(r)$ and $T_2(r)$ are any two characteristic functions of \eqref{lden}, where the coefficients are entire, then  $T_1(r)=T_2(r)+O(1)$. 
\end{lemma}

\begin{proof}
Let $\{f_1,\ldots,f_n\}$ and $\{g_1,\ldots,g_n\}$ be the solution bases for \eqref{lden} defining $T_1(r)$ and $T_2(r)$, respectively. Then there are constants $C_{i,j}\in\C$ such~that
	\begin{eqnarray*}
	f_1&=& C_{1,1}g_1+C_{1,2}g_2+\cdots+C_{1,n}g_n,\\
	&\vdots&\\
	f_n&=& C_{n,1}g_1+C_{n,2}g_2+\cdots+C_{n,n}g_n.
	\end{eqnarray*}
Denoting $C=\max\{|C_{i,j}|\}$, we obtain
	\begin{eqnarray*}
	|f_k|^2&\leq& 2\big(|C_{k,1}|^2|g_1|^2+\cdots+|C_{k,n}|^2|g_n|^2\big)\\
	&\leq& 2C^2\big(|g_1|^2+\cdots+|g_n|^2\big),\quad k=1,\ldots,n.
	\end{eqnarray*}
Hence, for $C_0\geq 0$ and $x\geq 0$, we may use
	\begin{eqnarray*}
	\log\sqrt{1+C_0x} &\leq& \frac12\log^+(1+C_0x)
	\leq \frac12\log^+x+O(1)\\
	&\leq& \frac12\log(1+x)+O(1)=\log\sqrt{1+x}+O(1)
	\end{eqnarray*}
to obtain $T_1(r)\leq T_2(r)+O(1)$. By changing the roles of the two fundamental bases, we obtain $T_2(r)\leq T_1(r)+O(1)$ similarly as above.
\end{proof}

\begin{theorem}[{\cite[Theorem~1]{BP}}]\label{th-1-BP}
Suppose that the coefficients $A_0,\ldots, A_{n-1}$  in \eqref{lden} are entire, and that $T(r)$ is a characteristic function of \eqref{lden}. If $f$ is a solution of \eqref{lden} for which 
    \begin{equation}\label{assu-Th1}
       \limsup_{r\to\infty}\frac{\log^m T(r)}{T(r,f)}<\infty 
    \end{equation}
is valid for some  $m>1$, then $f$ is a P-standard solution of \eqref{lden}.
\end{theorem}

\begin{remark}
It is noted in \cite[p.~1932]{BP} (or p.~1372 in the translation) that if the assumption \eqref{assu-Th1} is relaxed to
    \begin{equation}\label{alternative-assu}
    \limsup_{r\to\infty}\frac{\log T(r)\log ^{2+\varepsilon}T(r,f)}{T(r,f)}<\infty,
    \end{equation}
where $\varepsilon>0$ is arbitrary,
then the conclusion of Theorem~\ref{th-1-BP} remains valid. No further details about \eqref{alternative-assu} are given in \cite{BP}, though.
\end{remark} 
 
\begin{example}\label{example}
As discussed in \cite{BP} and originally observed by M.~Frei, the functions $f_1(z)=1+e^z$ and $f_2(z)=\exp\big(z+e^{-z}\big)$ are linearly independent solutions of
	$$
	f''+e^{-z}f'-f=0.
	$$
Since $\delta_P(1,f_1)=\pi$ and $\delta_N(1,f_1)=1$, the solution $f_1$ is neither P-standard nor N-standard. Moreover, since 
	$$
	T(r)\asymp T\big(r,\exp\big(e^z\big)\big)\sim \frac{e^r}{\sqrt{2\pi^3 r}}\quad\textnormal{and}\quad T(r,f_1)\sim r,
	$$ 
it follows that the exponent $m>1$ in \eqref{assu-Th1} cannot be replaced \mbox{with $m=1$.} Similarly, we see that the logarithmic term $\log ^{2+\varepsilon}T(r,f)$ in \eqref{alternative-assu} cannot be dropped, although we will later prove that the exponent $2+\varepsilon$ is not sharp. Note also that \eqref{Wittich-assumption} is not valid for $f=f_1$.
\end{example}

New results on standard solutions of \eqref{lden} are stated and discussed in Sections~\ref{NP-sec} and \ref{E-sec} below. Lemmas for the proofs are given in Section~\ref{lemmas-sec}, while the actual proofs can be found in Section~\ref{proofs-sec}. Section~\ref{concluding-sec} contains concluding remarks about Valiron deficient values.

%
%

\section{New results involving $T(r,f)$}\label{NP-sec}

Theorem~\ref{thefirst-mainthm} below shows that the $m$ solutions in Theorem~\ref{asymptotic-cor} are in fact P-standard. Thus the condition \eqref{2LM2} allows us to construct examples equations \eqref{lden} for which every nontrivial solution is a P-standard solution.

\begin{theorem}\label{thefirst-mainthm}
Under the assumptions of Theorem~\ref{asymptotic-cor}, every solution base of \eqref{lden} has $m\geq n-p$ P-standard solutions $f$ for which
\eqref{result-th1.3} holds.
\end{theorem}

\begin{remark}
The paper \cite{HYZ} contains more results in the spirit of Theorem~\ref{asymptotic-cor}. Using these results, more results in the spirit of Theorem~\ref{thefirst-mainthm} can be created. To avoid unnecessary repetition and to control the length of this paper, these discussions have been omitted.
\end{remark}

\medskip
The proof of Theorem~\ref{th-1-BP} in \cite{BP} is based on profound results in value distribution theory proved by Petrenko himself in his earlier papers.  Using much simpler methods, we are able to obtain an improvement of Theorem~\ref{th-1-BP}, which also improves \eqref{alternative-assu}. The main contribution of the following theorem, however, is to offer a proof that is simpler than that of Theorem~\ref{th-1-BP} in \cite{BP}.

\begin{theorem}\label{th-3-BP}
Suppose that the coefficients $A_0,\ldots, A_{n-1}$  in \eqref{lden} are entire, and that $T(r)$ is a characteristic function of \eqref{lden}. If $f$ is a solution of \eqref{lden} for which
    \begin{equation}\label{assu-Th3}
    \log T(r) \cdot \log^m \big(\log T(r)\big)=o(T(r,f)),\quad r\to\infty,\ r\not\in E, 
    \end{equation}
is valid for some  $m>1$, where $E\subset [0,\infty)$ satisfies $\ud(E)<1$, then $f$ is a P-standard solution of \eqref{lden}.
\end{theorem}

\begin{remark}
If the exponent $2+\varepsilon$ in \eqref{alternative-assu} is replaced with $1+\varepsilon$, we have
	$$
	\log T(r)=O\left(\frac{T(r,f)}{\log^{1+\varepsilon}T(r,f)}\right),
	$$
which in turn implies, for $m=1+\varepsilon/2$,
	$$
	\log T(r) \cdot \log^m \big(\log T(r)\big)
	=O\left(\frac{T(r,f)}{\log^{\varepsilon/2} T(r,f)}\right) = o(T(r,f)). 
	$$
The sharpness of \eqref{assu-Th3} is not known. However, \eqref{assu-Th3} is a milder assumption than \eqref{assu-Th1}, which in turn is relatively sharp by Example~\ref{example}.
\end{remark}

\medskip
We turn our attention to finding admissibility conditions in the spirit of \eqref{Wittich-assumption} for the solutions of \eqref{lden} to be P-standard.

\begin{theorem}\label{P-4.3}
Suppose that the coefficients $A_0,\ldots, A_{n-1}$ in \eqref{lden} are meromorphic, and that a meromorphic solution $f$ of \eqref{lden} satisfies
    \begin{eqnarray}
    \mathscr{L}(r,\infty,A_j) &=& o(T(r,f)),\quad r\not\in E,\, j=0,\ldots,n-1,\label{ass-A_i}\\
    \mathscr{L}(r,0,A_0) &=& o(T(r,f)),\quad r\not\in E,\label{ass-A_0}
    \end{eqnarray}
where $E\subset[0,\infty)$ satisfies $\ud (E)<1$.   
Then $0$ is the only possible finite Petrenko deviated value for $f$, i.e., $f$ is a P-standard solution.
\end{theorem}

Theorem~\ref{P-4.3} has the following consequence.

\begin{corollary}\label{polycoeff-cor}
Suppose that the coefficients $A_0,\ldots, A_{n-1}$ in \eqref{lden} are polynomials, and that $f$ is a non-trivial solution of \eqref{lden}. Then $f$ is a P-standard solution (and consequently an N-standard solution).
\end{corollary}

\begin{proof}
If $f$ is a polynomial, then clearly $\delta_P(a,f)=0$ for every $a\in\C$. Hence we may
suppose that $f$ is transcendental. But now the estimates in \eqref{ass-A_i} and \eqref{ass-A_0} are valid without an exceptional set. Thus $f$ is a P-standard solution by Theorem~\ref{P-4.3}.
\end{proof} 

\begin{example}\label{example2}
The function $f(z)=e^{2z}+1$ solves
	$$
	f''+A_1(z)f'+A_0(z)f=0,
	$$
where $A_0(z)=-2P(z)e^z$, $A_1(z)=P(z)e^z+P(z)e^{-z}-2$ and $P(z)$ is an arbitrary polynomial. Clearly
	$$
	\delta_P(1,f)\geq \delta_N(1,f)=1>0,
	$$
so that $f$ is neither a P-standard nor an N-standard solution. Since
	$$
	T(r,A_0)\asymp T(r,A_1)\asymp T(r,f)\asymp r,
	$$
	$$
	\mathscr{L}(r,0,A_0)\asymp \mathscr{L}(r,\infty,A_0)\asymp \mathscr{L}(r,\infty,A_1)
	\asymp r,
	$$
we see that $o(T(r,f))$ cannot be replaced with $O(T(r,f))$ in \mbox{\eqref{Wittich-assumption}, \eqref{ass-A_i}, \eqref{ass-A_0}.} 
\end{example}

An assumption analogous to \eqref{ass-A_0} for the characteristic function is not needed in  Theorem~\ref{N-standard-thm} because of the First Main Theorem. However, the assumption \eqref{ass-A_0} about $\mathscr{L}(r,0,A_0)$ can be replaced with an assumption about $\mathscr{L}(r,\infty,A_0)$ when $A_0$ is entire.

\begin{theorem}\label{alternative-thm}
The conclusion of Theorem~\ref{P-4.3} remains valid if $A_0$ is entire and, for some $m>1$, the assumption \eqref{ass-A_0} is replaced with
	\begin{equation}\label{ass-A_0B}
	\mathscr{L}(r,\infty,A_0)\cdot \log^m \mathscr{L}(r,\infty,A_0)= o(T(r,f)),\quad r\not\in E.
	\end{equation}
\end{theorem}

%
%

\section{New results involving $A(r,f)$}\label{E-sec} 
 
The results in the previous section are stated in terms of the characteristic function $T(r,f)$, where $f$ is either an entire function or meromorphic in $\C$. In this section we will discuss analogous results stated in terms of the function
	$$
	A(r,f)=\frac{1}{\pi}\int_{D(0,r)} f^\#(z)^2\, dm(z),
	$$
which is the normalized area of the image of the disc $D(0,r)$ on the Riemann sphere under $f$. Here $f^\#(z)=|f'(z)|/(1+|f(z)|^2)$ is the spherical derivative of $f$, and $dm(z)=sdsd\theta$ for $z=se^{i\theta}$ is the standard Euclidean area measure.

It is known that the functions $T(r,f)$ and $A(r,f)$ are connected. To see this, first recall that the Ahlfors-Shimizu characteristic given by
	$$
	T_0(r,f)=\int_0^r \frac{A(t,f)}{t}\, dt
	$$
satisfies  
	\begin{equation}\label{T-T0}
	T_0(r,f)=T(r,f)+O(1),
	\end{equation}
see \cite[p.~13]{Hayman} or \cite[p.~67]{Zheng}. It is easy to see that, for $C>1$,
	\begin{equation}\label{T0-A}
	\begin{split}
	T_0(r,f)&=\int_1^r \frac{A(t,f)}{t}\, dt+\int_0^1 \frac{A(t,f)}{t}\, dt
	\leq A(r,f)\log r+O(1),
	\end{split}	
	\end{equation}
	\begin{equation}\label{A-T0}
	A(r,f)=\frac{A(r,f)}{\log C}\int_r^{Cr}\frac{dt}{t}\leq \frac{1}{\log C}\int_r^{Cr}\frac{A(t,f)}{t}\, dt\leq \frac{T_0(Cr,f)}{\log C},
	\end{equation}
where we have used the fact that $A(r,f)$ is a non-decreasing function of $r$. For better estimates involving exceptional sets, see \cite[Lemma~2.4.2]{Zheng}.

In 1997, Eremenko \cite{Eremenko} introduced a quantity
	$$
	\delta_E(a,f)=\liminf_{r\to\infty}\frac{\mathscr{L}(r,a,f)}{A(r,f)},\quad a\in\widehat{\C}.
	$$
It is clear that $\delta_E(a,f)\geq 0$,
and it follows from an earlier theorem of Bergweiler and Bock \cite{BB} that if $\rho(f)\geq 1/2$, then $\delta_E(a,f)\leq \pi$. If $\delta_E(a,f)>0$, then $a$ is called a \emph{Eremenko deviated value} for $f$. Given a meromorphic function $f$, it is known \cite{Eremenko} that the set of E-deviated values for $f$ is at most countable and either consist of one point $a$ for which $\delta_E(a,f)>2\pi$ or
	\begin{equation}\label{E-def-rel}
	\sum_{a\in\widehat{\C}}\delta_E(a,f)\leq 2\pi.
	\end{equation}

We say that solution $f$ of \eqref{lden} satisfying  $\delta_E(a,f)=0$ for every $a\in\mathbb{C}\setminus\{0\}$ is an \emph{E-standard solution}. Note that the relationship between $\delta_E(a,f)$ and the quantities $\delta_N(a,f)$ and $\delta_P(a,f)$ is nontrivial \cite{Marchenko}.

\medskip
\begin{remark}
The proofs of the results in Section~\ref{NP-sec} use the fact that
	\begin{equation}\label{transcendental}
	\lim_{r\to\infty}\frac{T(r,f)}{\log r}=\infty
	\end{equation}
for any transcendental entire function $f$ \cite[p.~11]{YY}. In results involving E-standard solutions, \eqref{transcendental} should hold with $A(r,f)$ in place of $T(r,f)$. However, this is not always true. Indeed, it is known \cite{Clunie} that there are entire functions $f$ satisfying
	$$
	T_0(r,f)\sim T(r,f)\sim \log^2r,
	$$
in which case $A(r,f)\asymp \log r$ by \eqref{T0-A} and the following modification of \eqref{A-T0}:
	$$
	A(r,f)=\frac{A(r,f)}{\log r}\int_r^{r^2}\frac{dt}{t}\leq\frac{T_0(r^2,f)}{\log r}\sim 4\log r.
	$$
\end{remark}

\medskip
With the previous remark in mind, Lemma~\ref{A-lemma} below may be of independent interest. It partially relies on the concept of \emph{logarithmic order} \cite{Chern} of $f$ defined by
	 $$
	 \rho_{\log}(f):=\limsup_{r\to\infty}\frac{\log T(r,f)}{\log\log r}.
	 $$
If $f$ is a rational function, then $\rho_{\log}(f)=1$, while transcendental entire functions $g$ satisfying $\rho_{\log}(g)=1$ are known to exist \cite{Chern}. If a meromorphic function $h$ satisfies $\rho_{\log}(h)<1$, then $h$ is a constant function and $\rho_{\log}(h)=0$.

\begin{lemma}\label{A-lemma}
If $f$ is a transcendental meromorphic function, then $A(r,f)$ is an unbounded function of $r$ such that the following assertions hold.
\begin{itemize}
\item[\textnormal{(a)}]  If $\rho(f)>\alpha>0$, then the set
	$$
	H_1=\big\{r\geq 0: A(r,f)\geq r^\alpha\big\}
	$$
satisfies $\ud(H_1)=1$.
\item[\textnormal{(b)}] If $\rho_{\log}(f)>\alpha>2$, then the set
	$$
	H_2=\big\{r\geq 1: A(r,f)\geq \log^{\alpha-1} r\big\}
	$$
satisfies $\uld(H_2)=1$.
\end{itemize}
\end{lemma}

\begin{proof}
If $A(r,f)$ is bounded, then $T_0(r,f)=O(\log r)$, in which case $f$ is rational, and we have a contradiction. To prove (a), we define
	$$
	H_1^*=\big\{r\geq 1 : T(r,f)\geq 2r^\alpha\log r\big\},
	$$
which is essentially a subset of $H_1$ (modulo a bounded set) by \eqref{T-T0} and \eqref{T0-A}. Using \cite[Corollary~3.3]{HLWZ} with $\psi(r)=\log r$, we find that $\ud(H_1^*)=1$, and consequently $\ud(H_1)=1$. To prove (b), we define
	$$
	H_2^*=\big\{r\geq 1: T(r,f)\geq 2\log^{\alpha} r\big\},
	$$
which is essentially a subset of $H_2$. Similarly as above, using \cite[Corollary~3.3]{HLWZ} with $\psi(r)=\log\log r$, the assertion follows.
\end{proof}

Due to Lemma~\ref{A-lemma} and the remark preceding it, we have to pay attention to the validity of $\log r=o(A(r,f))$ when proving analogues of the results in Sections~\ref{Back-sec} and \ref{NP-sec}. We begin with an analogue of \mbox{Theorem~\ref{N-standard-thm}.} 

\begin{theorem}\label{E-analogue}
Suppose that the coefficients $A_0,\ldots, A_{n-1}$ in \eqref{lden} are meromorphic, and that $f$ is a meromorphic solution of \eqref{lden} such that 
    \begin{equation}\label{Eremenko-assumption}
    T(r,A_j)=o(A(r,f)),\quad r\not\in E,\, j=0,\ldots,n-1,
    \end{equation}
where $E\subset[0,\infty)$ satisfies $\ud(E)<1$. Then $0$ is the only possible finite E-deviated value for $f$, i.e., $f$ is an E-standard solution.
\end{theorem}

Theorem~\ref{P-4.3} has the following analogue, where $o(T(r,f))$ is being replaced with $o(A(r,f))$. 

\begin{theorem}\label{E1}
Suppose that the coefficients $A_0,\ldots, A_{n-1}$ in \eqref{lden} are meromorphic, and that a meromorphic solution $f$ of \eqref{lden} satisfies $\rho_{\log}(f)>2$ and
    \begin{eqnarray}
    \mathscr{L}(r,\infty,A_j) &=& o(A(r,f)),\quad r\not\in E,\; j=0,\ldots,n-1,\label{ass-A_i2}\\
    \mathscr{L}(r,0,A_0) &=& o(A(r,f)),\quad r\not\in E,\label{ass-A_02}
    \end{eqnarray}
where $E\subset[0,\infty)$ satisfies $\ud(E)<1$.   
Then $0$ is the only possible finite E-deviated value for $f$, i.e., $f$ is an E-standard solution.
\end{theorem}

\begin{remark}
(1) The solution $f$ in Example~\ref{example2} is not an E-standard solution. Hence Example~\ref{example2} can be used to illustrate that  $o(A(r,f))$ cannot be replaced with $O(A(r,f))$ in \eqref{Eremenko-assumption}, \eqref{ass-A_i2}, \eqref{ass-A_02}.

(2) If the coefficients $A_0,\ldots,A_{n-1}$ are entire, then the technical assumption $\rho_{\log}(f)>2$ in Theorem~\ref{E1} can be omitted. Indeed, the proof could then be handled in two cases: (i) At least one of the coefficients is non-constant, or (ii) all coefficients are constant functions. See the proof of Theorem~\ref{E-analogue} in Section~\ref{proofs-sec} below for an analogous reasoning.
\end{remark}

\medskip
The next result is an analogue of Corollary~\ref{polycoeff-cor} for E-standard solutions.

\begin{corollary}\label{polycoeff-cor2}
Suppose that the coefficients $A_0,\ldots, A_{n-1}$ in \eqref{lden} are polynomials, and that $f$ is a non-trivial solution of \eqref{lden}. Then $f$ is an E-standard solution.
\end{corollary}

\begin{proof}
Suppose first that $f$ is a polynomial. Then, since $T_0(r,f)\asymp \log r$, we see from \eqref{T0-A} that $A(r,f)$ is bounded away from zero. Clearly, for every $a\in\C$, we have $\mathscr{L}(r,a,f)\to 0$ as $r\to\infty$, so that $\delta_E(a,f)=0$. 

We now suppose that $f$ is transcendental. Then it is known that $\rho(f)\geq 1/(n-1)$ \cite{GSW}. Let $\alpha\in (0,1/(n-1))$, and let $H_1$ be the set in Lemma~\ref{A-lemma}(a). Since the coefficients are polynomials, it follows that the estimates in \eqref{ass-A_i2} and \eqref{ass-A_02} are valid for all $r\in H_1$ (as opposed to for all $r\not\in E$). The assertion follows by a careful inspection of the proof of Theorem~\ref{P-4.3}.
\end{proof}

Further analogues of the results in Sections~\ref{Back-sec} and~\ref{NP-sec} for E-standard solutions can be obtained by replacing $o(T(r,f))$ in the assumptions with $o(A(r,f))$, modulo minor technical adjustments. The details are omitted.

%
%
 
\section{Lemmas}\label{lemmas-sec}

In this section we discuss lemmas which are either new or non-trivial modifications of existing results, or which need to be clarified to the reader in some way. The proofs of the main results also rely on lemmas that can be found directly from the literature. Such lemmas will not be stated here.  

The following generalization of Borel's lemma is the key to most of the crucial estimates in this paper.

\begin{lemma}[{\cite[Lemma~3.3.1]{CY}}]\label{B-lemma}
Let $F(r)$ and $\phi(r)$ be a positive, nondecreasing and continuous functions defined for $r_0\leq r<\infty$, and assume that $F(r)\geq e$ for $r\geq r_0$. Let $\xi(r)$ be a positive, nondecreasing and continuous functions defined for $e\leq r<\infty$. Finally, let $C>1$ be a constant, and let $E\subset [r_0,\infty)$ be defined by
	$$
	E=\left\{r\geq r_0 : F\left(r+\frac{\phi(r)}{\xi(F(r))}\right)\geq CF(r)\right\}.
	$$
Then, for all $R\in (r_0,\infty)$,
	\begin{equation}\label{B-assertion}
	\int_{E\cap [r_0,R]}\frac{dr}{\phi(r)}
	\leq \frac{1}{\xi(e)}+\frac{1}{\log C}\int_e^{F(R)}\frac{dx}{x\xi(x)}.
	\end{equation}
\end{lemma}

\medskip
\begin{remark}
It follows from the proof of Lemma~\ref{B-lemma} in \cite{CY} that if the set $E$ is unbounded, then the function $F(r)$ must be unbounded. Moreover, the set $E$ is contained in a union of closed intervals
	\begin{equation}\label{F}
	E\subset \cup_{\nu=1}^\infty [r_\nu,s_\nu]=:F.
	\end{equation}
If $\xi$ is chosen such that 
	\begin{equation}\label{xi}
	\int_e^\infty \frac{dx}{x\xi(x)}<\infty,
	\end{equation}
then we may let $R\to\infty$ in \eqref{B-assertion}, and find, again by the proof of Lemma~\ref{B-lemma} in \cite{CY}, that
	\begin{equation}\label{EF-finite-meas}
	\int_E\frac{dr}{\phi(r)}\leq \int_F\frac{dr}{\phi(r)}\leq \sum_{\nu=1}^\infty
	\int_{r_\nu}^{s_\nu}\frac{dr}{\phi(r)}<\infty.
	\end{equation}
In the special cases when $\phi(r)\equiv 1$ or $\phi(r)=r$, the set $E$ has finite linear/logarithmic measure, respectively. The sharpest choice for $\xi$ satisfying \eqref{xi} is the Khinchin function, for which the condition \eqref{xi} is used as part of the definition \cite[p.~49]{CY}. Another natural choice is $\xi(x)=\log^mx$ for $m>1$.
\end{remark} 

\medskip
We need an estimate for the number of zeros of an entire function. The difference to the standard estimate in  \cite[p.~15]{Levin} is that the function in the upper bound has $r$ as a variable instead of $\alpha r$, where $\alpha>1$ is a constant.

\begin{lemma}\label{n-lemma}
Let $g$ be a nonconstant entire function with $g(0)=1$,  let $n(r)$ denote the number of zeros of $g$ in $|z|\leq r$, and let $m>1$. Then
	$$
	n(r)\lesssim \mathscr{L}(r,\infty,g)\cdot \log^m \mathscr{L}(r,\infty,g),
	\quad r\not\in [0,1]\cup E,
	$$
where $E\subset(1,\infty)$ has finite logarithmic measure.
\end{lemma}

\begin{proof}
If $g$ is a polynomial, then $n(r)$ is a bounded function, and the assertion is clear even
without an exceptional set. Thus we may suppose that $g$ is transcendental, or, in fact,
that $n(r)$ is an unbounded function. 

Let $1<r<\infty$, and let $R\in (r,\infty)$ be arbitrary. Analogously as in \cite[p.~15]{Levin}, we use Jensen's theorem and the monotonicity of $n(r)$ for
	\begin{eqnarray*}
	n(r) &=& \frac{n(r)}{\log (R/r)}\int_r^R\frac{dt}{t}\leq 
	\frac{1}{\log (R/r)}\int_0^R\frac{n(t)}{t}\, dt\\
	&= & \frac{1}{2\pi \log (R/r)}\int_0^{2\pi}\log^+|g(Re^{i\theta})|\, d\theta
	\leq \frac{\mathscr{L}(R,\infty,g)}{\log (R/r)}.
	\end{eqnarray*}
Since $\log x>(x-1)/2$ for $1<x<2$, we obtain
	\begin{equation}\label{n-estimate}
	n(r)\leq \frac{2r}{R-r}\cdot\mathscr{L}(R,\infty,g),
	\quad r<R<2r.
	\end{equation}
We now make a specific choice
	\begin{equation}\label{R}
	R=r+\frac{r}{\log^m \mathscr{L}(r,\infty,g)},
	\end{equation}
which satisfies $r<R<2r$ whenever $r>1$ is large
enough, say $r\geq r_0$. Since the maximum modulus $M(r,g)$ is a continuous function of $r$ \cite[p.~2]{Levin}, it follows that $\mathscr{L}(r,\infty,g)$ is continuous. The assertion now follows from \eqref{n-estimate} by applying \eqref{R} and Lemma~\ref{B-lemma} with the choices $\phi(r)=r$, $\xi(r)=\log^m r$ and $F(r)=\mathscr{L}(r,\infty,g)$. 
\end{proof}

Lemma~\ref{minimum-modulus-lemma} below is an estimate for the minimum modulus of an entire function, and
is a modification of Theorem~11 in \cite[p.~21]{Levin}. The difference to the estimate
in \cite[p.~21]{Levin} is that the function in the lower bound has $r$ as a variable instead of
$\alpha r$, where $\alpha>1$. 

\begin{lemma}\label{minimum-modulus-lemma}
Let $g$ be a transcendental entire function with $g(0)=1$, and let $\delta>0$ and $m>1$. Then there exists a set $F\subset[0,\infty)$ with $\ud(F)<\delta$ such that
	$$
	\log |g(z)|\gtrsim -\left(1+\log\frac{1}{\delta}\right)\mathscr{L}(r,\infty,g)\cdot 
	\log^m \mathscr{L}(r,\infty,g),\quad |z|=r\not\in F.
	$$
\end{lemma}

\begin{proof}
Let $2<r_0<R<\infty$ be such that $\mathscr{L}(r_0,\infty,g)\geq 1$.
If $g$ has no zeros, then \cite[p.~19]{Levin} yields
	$$
	\log |g(z)|\geq -\frac{2r}{R-r}\mathscr{L}(R,\infty,g),\quad r_0<|z|=r<R.
	$$
Choosing $R$ as in \eqref{R}, we find that there exists a set $E_1\subset (1,\infty)$ of finite logarithmic measure such that	
	$$
	\log |g(z)|\gtrsim -\mathscr{L}(r,\infty,g)\cdot \log^m \mathscr{L}(r,\infty,g),
	\quad r\not\in [0,1]\cup E_1,\ r>r_0.
	$$
We have $\ud(E_1)=0$ by \cite[p.~9]{Zheng}, and
hence, from now on, we may suppose that $g$ has zeros. 
	
Again for arbitrary $2<r_0<R<\infty$ such that $\mathscr{L}(r_0,\infty,g)\geq 1$, we define
	$$
	\varphi(z)=\frac{(-R)^n}{a_1\cdots a_n}\prod_{k=1}^n \frac{R(z-a_k)}{R^2-\bar{a}_kz},
	$$
where the points $a_1,\ldots,a_n$ are the zeros of $g$ in the closed disc $\overline{D}(0,R)$. 
Note that $a_j\neq 0$ for all $j=1,\ldots,n$ by the assumption $g(0)=1$. We have
	$$
	\varphi(0)=1\quad\textnormal{and}\quad
	|\varphi(Re^{i\theta})|=\frac{R^n}{|a_1\cdots a_n|}\geq 1.
	$$
The function $\psi(z)=g(z)/\varphi(z)$ has no zeros in $\overline{D}(0,R)$, and hence, by
\cite[p.~19]{Levin}, we have
	\begin{equation*}
	\begin{split}
	\log |\psi(z)| &\geq -\frac{2r}{R-r}\mathscr{L}(R,\infty,\psi)\\
	&\geq -\frac{2r}{R-r}\mathscr{L}(R,\infty,g)+\frac{2r}{R-r}\log \frac{R^n}{|a_1\cdots a_n|}\\
	&\geq -\frac{2r}{R-r}\mathscr{L}(R,\infty,g),\quad r_0<|z|=r<R.
	\end{split}
	\end{equation*}
Choosing $R$ as in \eqref{R} and deducing similarly as above,
	\begin{equation}\label{estimate-psi}
	\log |\psi(z)|\gtrsim -\mathscr{L}(r,\infty,g)\cdot \log^m \mathscr{L}(r,\infty,g),
	\quad r\not\in [0,1]\cup E_1,\ r>r_0,
	\end{equation}
where $E_1\subset (1,\infty)$ has finite logarithmic measure.

We proceed to estimate $|\varphi(z)|$ from below. Clearly,
	$$
	\prod_{k=1}^n \left|R^2-\bar{a}_kz\right|\leq (2R^2)^n,\quad |z|< R.
	$$
Let $\eta>0$ be a small constant. Applying Cartan's lemma as in \cite[p.~22]{Levin},
there exists a union $\U(R)$ of Euclidean discs with the following properties: The sum of the radii
of the discs in $\U(R)$ equals $4\eta R$, and
	$$
	\left|\prod_{k=1}^n R(z-a_k)\right|\geq R^n\left(\frac{2\eta R}{e}\right)^n,
	\quad z\not\in\U(R).
	$$
Consequently,
	$$
	|\varphi(z)|\geq \left(\frac{\eta}{e}\right)^n,\quad |z|=r<R,\ z\not\in \U(R).
	$$
Let $F(R)\subset [0,\infty)$ denote the circular projection of $\U(R)$. As $R\to\infty$,
the set $F(R)$ approaches to a set $F_1$ that satisfies $\ud(F_1)\leq 8\eta$.	Then
	$$
	\log |\varphi(z)|\geq n\log \left(\frac{\eta}{e}\right),\quad |z|=r<R,\ r\not\in F_1.
	$$
By Lemma~\ref{n-lemma},
	$$
	n=n(R)\lesssim \mathscr{L}(R,\infty,g)\cdot \log^m \mathscr{L}(R,\infty,g),
	\quad R\not\in [0,1]\cup E_2,
	$$
where $E_2\subset(1,\infty)$ has finite logarithmic measure. For the choice of $R$ in \eqref{R}, define a set
	\begin{equation}\label{E3}
	E_3:=\{r>r_0: R\in E_2\}.
	\end{equation}
Postponing the proof that $E_3$ has finite logarithmic measure, the set $E:=E_1\cup E_2\cup E_3$ has finite logarithmic measure also, and
	$$
	n=n(R)\lesssim \mathscr{L}(r,\infty,g)\cdot \log^m \mathscr{L}(r,\infty,g),
	\quad r\not\in [0,r_0]\cup E.
	$$
It is known that a set of finite logarithmic measure has zero upper linear density \cite[p.~9]{Zheng}. Hence, by choosing $\eta=\delta/8$, the set $F:=[0,r_0]\cup E\cup F_1$
satisfies $\ud(F)\leq \delta$.	Moreover,
	\begin{equation}\label{varphi-down}
	\log |\varphi(z)|\gtrsim -\left(1+\log\frac{1}{\delta}\right) \mathscr{L}(r,\infty,g)\cdot \log^m \mathscr{L}(r,\infty,g),
	\quad r\not\in F.
	\end{equation}
The assertion follows by applying \eqref{estimate-psi} and \eqref{varphi-down} to
$\psi(z)=g(z)/\varphi(z)$.

It remains to prove that the set $E_3$ in \eqref{E3} has finite logarithmic measure. This is clearly the case if $E_2$ is a bounded set, so we suppose that $E_2$ is unbounded. The set $E_2$ comes from Lemma~\ref{n-lemma}, and hence from Lemma~\ref{B-lemma}. Thus it can be
covered by a union of closed intervals $F$ as in \eqref{F} giving us
	$$
	E_3\subset \{r>r_0: R\in F\}.
	$$
Since $R=R(r)$ is a continuous function in $r$, the pre-image of every closed interval
$[r_\nu,s_\nu]$ constituting the set $F$ is a closed interval, say $[\alpha_\nu,\beta_\nu]$.
It follows that 
	$$
	E_3\subset\cup_{\nu=1}^\infty [\alpha_\nu,\beta_\nu].
	$$
From \eqref{EF-finite-meas} with $\phi(r)=r$, we have
	$$
	\sum_{\nu=1}^\infty \log\frac{s_\nu}{r_\nu}<\infty.
	$$
Since $R(r)/r\to 1$ as $r\to\infty$, we have
	$$
	\log\frac{s_\nu}{r_\nu}\cdot\left(\log\frac{\beta_\nu}{\alpha_\nu}\right)^{-1}\to 1,\quad \nu\to\infty,
	$$
and consequently
	$$
	\int_{E_3}\frac{dr}{r}\leq \sum_{\nu=1}^\infty \int_{\alpha_\nu}^{\beta_\nu}\frac{dr}{r}
	=\sum_{\nu=1}^\infty \log\frac{\beta_\nu}{\alpha_\nu}<\infty
	$$
by the Limit Comparison Test.
\end{proof}

The next lemma follows from \cite[Lemma~5.1]{HYZ} and \cite[Remark~5.2]{HYZ}. 
  
\begin{lemma}[{\cite{HYZ}}]\label{LD-lemma}
Let $f$ be a meromorphic function in $\C$, and suppose that $k,j$ are integers with $k>j\geq0$
and $f^{(j)}\not\equiv 0$. Then there exists a set $E\subset [0,\infty)$ of finite linear
measure such that
     \begin{equation}\label{LD-conclusion}
     \log^+\left|\frac{f^{(k)}(z)}{f^{(j)}(z)}\right|\lesssim 
     \log T(r, f)+\log r,\quad |z|=r\not\in E.
     \end{equation} 
 \end{lemma}

\medskip
Finally, we need a lemma which is in the spirit of Frank and Hennekemper \cite[Lemma~7.7]{Laine} and Petrenko \cite[Corollary]{BP}.

\begin{lemma}\label{FHP-lemma}
Let $f_1,\ldots,f_n$ be linearly independent meromorphic solutions of \eqref{lden} with meromorphic coefficients $A_0,\ldots ,A_{n-1}$. Then there exists a set $E\subset [0,\infty)$ of finite linear measure such that for every $j=0,\ldots,n-1$,
	\begin{equation}\label{FHP-conclusion}
	\mathscr{L}(r,\infty,A_j)=O\left(\log r +\max_{1\leq j\leq n}\log T(r,f_j)\right),
	\quad r\not\in E.
	\end{equation}
\end{lemma}

The assertion is easy to verify in the special case $f^{(n)}+A(z)f=0$, as one just needs
to write $|A(z)|=|f^{(n)}(z)/f(z)|$ and apply Lemma~\ref{LD-lemma}.
The proof of the general case is by induction and follows closely that of \cite[Lemma~7.7]{Laine}. Each time the growth of a logarithmic derivative is to be estimated, one should use Lemma~\ref{LD-lemma}. The details are omitted.

\medskip
\begin{remark}
If the coefficients in \eqref{lden} are entire functions, then the conclusion \eqref{FHP-conclusion} of Lemma~\ref{FHP-lemma} can be written alternatively as 
	\begin{equation}\label{FHP-conclusion2}
	\mathscr{L}(r,\infty,A_j)=O\left(\log r +\log T(r)\right),
	\quad r\not\in E,\ j=0,\ldots,n-1.
	\end{equation}
This is a simple consequence of the facts that $T(r,f_j)=m(r,f_j)$ and
	\begin{equation}\label{characteristic-solutionbase}
	\log^+|f_j|\leq \log\sqrt{1+\sum_{k=1}^n |f_k|^2}
	\leq \sum_{k=1}^n \log^+|f_k|+\frac12\log(n+1)
	\end{equation}
for every $j=1,\ldots,n$.
\end{remark}

\medskip
From \eqref{characteristic-solutionbase} it is obvious that at least one of the functions in a solution base of \eqref{lden} is of infinite order if and only if the corresponding characteristic function $T(r)$ of \eqref{lden} is of infinite order. This gives raise to the following version of a well-known result of Frei.

\begin{lemma}[{\cite{Frei-1953},\cite[Theorem~4.2]{Laine}}]\label{Frei-thm}
Suppose that the coefficients $A_0,\ldots, A_{n-1}$ in \eqref{lden} are entire, and that at least one of them is transcendental. Then any characteristic function $T(r)$ of \eqref{lden} must be of infinite order of growth.
\end{lemma}

%
%

\section{Proofs of theorems}\label{proofs-sec}

\noindent
\emph{Proof of Theorem~\ref{thefirst-mainthm}.}
Let $\{f_1,\ldots,f_n\}$ be a solution base for \eqref{lden}, and let $T(r)$ be the associated characteristic function of \eqref{lden} defined in \eqref{characteristic-function}. From \eqref{characteristic-function}, \eqref{characteristic-solutionbase} and \cite[Corollary~5.3]{HKR},
	$$	
	T(r)\leq \sum_{k=1}^n m(r,f_k)+O(1)
	\lesssim  r\sum_{j=0}^{n-1}M(r,A_j)^{1/(n-j)}+1.
	$$
Thus
	$$
	\log^+T(r)\leq \sum_{j=0}^{n-1}\log^+ M(r,A_j)+O(\log r).
	$$
In the course of proof of \cite[Theorem~2.3]{HYZ}, it is observed that
	\begin{equation}\label{Canada1}
	\log M(r,A_j)\lesssim \log M(r,A_p),\quad j=0,\ldots,n-1,
	\end{equation}
and that at least $n-p$ solutions $f$ in the solution base $\{f_1,\ldots,f_n\}$ satisfy
	\begin{equation}\label{Canada2}
	\log M(r,A_p)\lesssim \log T(r,f),\quad r\not\in E,
	\end{equation}
where $E\subset [0,\infty)$ has finite linear measure. Let $f$ be a solution of \eqref{lden} that satisfies \eqref{Canada2}. Using \eqref{Canada1}, \eqref{Canada2} and the fact that $A_p$ is transcendental, we obtain
	$$
	\log T(r)\lesssim \log M(r,A_p)\lesssim \log T(r,f),\quad r\not\in E.
	$$
Then $f$ is a P-standard solution by (yet to be proved) Theorem~\ref{th-3-BP}.\qed

\medskip
\noindent 
\emph{Proof of Theorem~\ref{P-4.3}.}
The main idea for the proof is from Wittich \cite[p.~54]{Wittich}, which is generalized by Laine in \cite[p.~62]{Laine}. To estimate the logarithmic derivatives, we make use of Lemma~\ref{LD-lemma} from the previous section.

If $f$ is a polynomial, then clearly $\delta_P(a,f)=0$ for every $a\in\C$. Hence we may
suppose that $f$ is transcendental.
It follows from \eqref{lden} that
    \begin{equation*}
     \frac{1}{f-a}=-\frac{1}{aA_0}\left(A_0+A_1\frac{(f-a)'}{f-a}+\cdots+A_n\frac{(f-a)^{(n)}}{f-a}\right)
    \end{equation*}
for any $a\in\C\setminus\{0\}$. Then
     \begin{equation*}
     \begin{split}
     \log^+\frac{1}{|f(z)-a|}&\leq \log^+\frac{1}{|A_0(z)|}+\sum_{j=0}^n\log^+|A_j(z)|+\sum_{j=1}^n\log^+\left|\frac{(f(z)-a)^{(j)}}{f(z)-a}\right|+O(1)\\
     &\leq \mathscr{L}(r,0,A_0)+\sum_{j=0}^n \mathscr{L}(r,\infty,A_j)
     +\sum_{j=1}^n\log^+\left|\frac{(f(z)-a)^{(j)}}{f(z)-a}\right|+O(1).
     \end{split}
     \end{equation*}
Thus it follows from Lemma~\ref{LD-lemma} and the assumptions \eqref{ass-A_i} and \eqref{ass-A_0} that
    \begin{equation}\label{L-L-T-loglog r}
     \log^+\frac{1}{|f(z)-a|}=o(T(r,f)),\quad |z|=r\not\in E_1,
     \end{equation}
where $E_1\subset[0,\infty)$ is the set that consists of the set $E$ appearing in \eqref{ass-A_i} and \eqref{ass-A_0} as well as of the set $E$ appearing in Lemma~\ref{LD-lemma}. Since any set $F\subset [0,\infty)$ of finite linear measure satisfies $\ud (F)=0$, we have $\ud (E_1)<1$. Observing that the right-hand side of \eqref{L-L-T-loglog r} does not depend on the argument of $z$, we deduce that
	$$
	\mathscr{L}(r,a,f)=o(T(r,f)),\quad r\not\in E_1.
	$$
This implies
	$$
	\liminf_{r\to\infty}\frac{\mathscr{L}(r,a,f)}{T(r,f)}
	\leq \underset{r\not\in E_1}{\liminf_{r\to\infty}}\frac{\mathscr{L}(r,a,f)}{T(r,f)}
	\leq \limsup_{r\to\infty}\frac{o(T(r,f))}{T(r,f)}=0,
	$$
that is, $\delta_P(a,f)=0$, and the assertion is now proved. \qed

\medskip
\noindent 
\emph{Proof of Theorem \ref{alternative-thm}.}
Let $f$ be a nontrivial solution of \eqref{lden} satisfying \eqref{ass-A_i} and \eqref{ass-A_0B}.
By the proof of Theorem~\ref{P-4.3}, we may suppose that $f$ is transcendental. If $A_0$
is a polynomial, then $\mathscr{L}(r,0,A_0)=O(1)$, and consequently \eqref{ass-A_0} is valid. The assertion then follows by Theorem~\ref{P-4.3}. Therefore, we may suppose that $A_0$ is transcendental. 

Suppose that $A_0(0)=1$, and choose $\delta>0$ such that $\ud(E)+\delta<1$, where $E$ is the set in \eqref{ass-A_i}. Then, applying Lemma~\ref{minimum-modulus-lemma} to $A_0$, we have
	$$
	\mathscr{L}(r,0,A_0)\lesssim \mathscr{L}(r,\infty,A_0)\log^m \mathscr{L}(r,\infty,A_0),
	\quad r\not\in F,
	$$
where $m>1$ and $F\subset[0,\infty)$ satisfies $\ud(F)<\delta$. The assumption \eqref{ass-A_0B} then gives
	$$
	\mathscr{L}(r,0,A_0)=o(T(r,f)),\quad r\not\in G,
	$$
where $G=E\cup F$ satisfies $\ud(G)\leq \ud(E)+\ud(F)<1$. The assertion follows by the proof of Theorem~\ref{P-4.3}. 

If $A_0(0)\neq 1$, we may find constants $C\in\C$ and $k\in\Z$ such that the function $B_0(z)=Cz^kA_0(z)$ is entire and satisfies $B_0(0)=1$. 
Then Lemma~\ref{minimum-modulus-lemma} can be applied to~$B_0$. Moreover, since $A_0$
is transcendental, we have 
	$$
	\mathscr{L}(r,0,B_0)=\mathscr{L}(r,0,A_0)+O(\log r)
	=(1+o(1))\mathscr{L}(r,0,A_0).
	$$
The assertion now follows similarly as in the case $A_0(0)=1$. \qed

\medskip
\noindent
\emph{Proof of Theorem~\ref{th-3-BP}.}
Let $\{f_1,\ldots,f_n\}$ be a solution base for \eqref{lden}, and let $T(r)$ be the associated characteristic function of \eqref{lden} defined in \eqref{characteristic-function}. We are given that a certain solution $f$ of \eqref{lden} satisfies \eqref{assu-Th3}. As above, we may suppose that $f$ is transcendental.

Suppose first that all of the coefficients $A_0,\ldots ,A_{n-1}$ in \eqref{lden} are polynomials.  Then $f$ is a P-standard solution by Corollary~\ref{polycoeff-cor}. Note that the assumption \eqref{assu-Th3} is not needed in this particular case. 

Suppose then that at least one of the coefficients
of \eqref{lden} is transcendental, in which case $T(r)$ is of infinite order by Lemma~\ref{Frei-thm}.  Thus, from \eqref{assu-Th3}, \eqref{FHP-conclusion2}, and the fact that $f$ is transcendental,
	\begin{equation}\label{T-inf-order}
	\begin{split}
	\mathscr{L}(r,\infty,A_j)&= O(\log r+\log T(r))=O(\log r)+o(T(r,f))\\
	&=o(T(r,f)),
	\quad r\not\in E_2,\ j=0,\ldots,n-1,
	\end{split}	
	\end{equation}
where $E_2$ is the union of the exceptional sets appearing in \eqref{assu-Th3} and \eqref{FHP-conclusion2}, and hence satisfies $\ud(E_2)<1$.	
This implies \eqref{ass-A_i} for the specific solution~$f$. We proceed to prove that \eqref{ass-A_0} holds for the same solution $f$.

If $A_0$ is a polynomial, then 	\eqref{ass-A_0} holds without an exceptional set because
$f$ is transcendental, and consequently the assertion follows by Theorem~\ref{P-4.3}. Hence we may suppose that $A_0$ is transcendental. Let $\alpha>0$ be arbitrarily large but fixed. By \cite[Corollary~3.7]{LHZ}, the set
	$$
	G=\{r\geq 0 : T(r)\geq r^\alpha\}
	$$
satisfies $\ud(G)=1$. Choose $\delta>0$ small enough such that $\ud(E_2)+\delta<1$, where $E_2$ is the exceptional set in \eqref{T-inf-order}, and hence contains the exceptional sets in \eqref{assu-Th3} and \eqref{FHP-conclusion2}.
Now, if $A_0(0)=1$, it follows from Lemma~\ref{minimum-modulus-lemma}, \eqref{assu-Th3} and \eqref{FHP-conclusion2} that
	\begin{equation}\label{A0}
	\begin{split}
	\mathscr{L}(r,0,A_0) &= O\left(\mathscr{L}(r,\infty,A_0)\cdot 
	\log^m \mathscr{L}(r,\infty,A_0)\right)\\
	&=O\big(\log T(r)\cdot 
	\log^m (\log T(r))\big)\\
	&=o(T(r,f)), \quad r\in G\setminus (E_2\cup F),
	\end{split}
	\end{equation}
where $\ud(F)<\delta$. The set $H=G\setminus (E_2\cup F)$ satisfies
	\begin{equation*}
	\begin{split}
	\ud(H) & \geq  \limsup_{r\to\infty}\frac{\int_{G\cap [0,r]}dt-\int_{(E_2\cup F)\cap [0,r]}dt}{r}\\
	& \geq  \limsup_{r\to\infty}\frac{\int_{G\cap [0,r]}dt}{r}+
	\liminf_{r\to\infty}\left(-\frac{\int_{(E_2\cup F)\cap [0,r]}dt}{r}\right)\\
	&= \ud(G)-\ud(E_2\cup F)= 1-\delta>0.
	\end{split}
	\end{equation*}
The estimate in \eqref{A0} is the same as the estimate in \eqref{ass-A_0}, but is valid outside of a different exceptional set. 
By carefully studying the proof of Theorem~\ref{P-4.3} and using \eqref{T-inf-order}, it follows that $f$ is a P-standard solution. The case $A_0(0)\neq 1$ can be dealt with similarly as in the proof of Theorem~\ref{alternative-thm}.  \qed

\medskip
\noindent
\emph{Proof of Theorem~\ref{E-analogue}.}
Suppose that at least one of the coefficients $A_j$ is non-constant. Then $A_j$ is either a non-constant rational function or a transcendental meromorphic function. In both cases, the assumption \eqref{Eremenko-assumption} implies that
	$$
	\log r=o(A(r,f)),\quad r\not\in E.
	$$
Moreover, using \eqref{T-T0} and \eqref{T0-A}, the conclusion of Lemma~\ref{LD-lemma} can be re-written as
	\begin{equation*}
	\log^+\left|\frac{f^{(k)}(z)}{f^{(j)}(z)}\right|\lesssim 
     \log A(r, f)+\log r,\quad |z|=r\not\in E_1,
	\end{equation*}
where $E_1\subset[0,\infty)$ has finite linear measure. Denote $F=[0,\infty)\setminus (E\cup E_1)$. Then $\ud(F)>0$ and 	
	\begin{equation}\label{LD-conclusion2}
	\log^+\left|\frac{f^{(k)}(z)}{f^{(j)}(z)}\right|=o(A(r, f)),\quad |z|=r\in F.
	\end{equation}	
Deducing similarly as in the proof of \cite[Theorem~4.3]{Laine} (alternatively, see the proof of Theorem~\ref{P-4.3}), operating in the set $F$ and using \eqref{LD-conclusion2} every time an estimate is needed for logarithmic derivatives, it follows that $\delta_E(a,f)=0$ for every $a\in\mathbb{C}\setminus\{0\}$. This implies the assertion. 

Suppose then that all coefficients $A_j$ are constant functions. Then it is well known that $f$ is a linear combination of terms of the form $z^ke^{\alpha z}$, where $k\geq 0$ is an integer and $\alpha\in\C\setminus\{0\}$ is a root of the associated characteristic equation. Then $T(r,f)=(C+o(1))r$ for some constant $C>0$ \cite[Satz~1]{Stein}. 
For $a\in\mathbb{C}\setminus\{0\}$, we have
	$$
	m(r,a,f)=o(r),\quad r\to\infty,
	$$
by \cite[Satz~2]{Stein}, and hence the Valiron deficiency given by
	\begin{equation}\label{Valiron-def}
	\delta_V(a,f):=\limsup_{r\to\infty}\frac{m(r,a,f)}{T(r,f)}
	\end{equation}
satisfies $\delta_V(a,f)=0$. Since $f$ has lower order $\mu(f)=1$, it follows that
	$$
	0\leq \delta_E(a,f)\leq \pi\sqrt{\delta_V(a,f)(2-\delta_V(a,f))}=0,
	$$
see \cite[Theorem~L]{Marchenko} or \cite[Theorem~2]{Marchenko2}. This implies the assertion. \qed

\medskip
\noindent
\emph{Proof of Theorem~\ref{E1}.}
By Lemma~\ref{A-lemma}(b), we have
	\begin{equation*}
	\underset{r\in H_2}{\lim_{r\to\infty}}\frac{A(r,f)}{\log r}=\infty
	\end{equation*}
where the set $H_2\subset [0,\infty)$ satisfies $\ud(H_2)=1$. Therefore, the identity in \eqref{LD-conclusion2} holds for $F=H_2\setminus (E\cup E_1)$, which satisfies
	\begin{equation*}
	\begin{split}
	\ud(F) & \geq  \limsup_{r\to\infty}\frac{\int_{H_2\cap [0,r]}dt-\int_{(E\cup E_1)\cap [0,r]}dt}{r}\\
	& \geq  \limsup_{r\to\infty}\frac{\int_{H_2\cap [0,r]}dt}{r}+
	\liminf_{r\to\infty}\left(-\frac{\int_{(E\cup E_1)\cap [0,r]}dt}{r}\right)\\
	&= \ud(H_2)-\ud(E\cup E_1)= 1-\ud(E)>0.
	\end{split}
	\end{equation*}
Deducing similarly as in the proof of \cite[Theorem~4.3]{Laine} or as in the proof of Theorem~\ref{P-4.3}, by operating in the set $F$ and using \eqref{LD-conclusion2}, it follows that $\delta_E(a,f)=0$ for every $a\in\mathbb{C}\setminus\{0\}$. \qed

%
%

\section{Concluding remarks}\label{concluding-sec}

Recall the Valiron deficiency $\delta_V(a,f)$ for a meromorphic function $f$ from \eqref{Valiron-def}.
If $\delta(a,f)>0$ for $a\in\widehat{\C}$, then $a$ is called a \emph{Valiron deficient value} for $f$. 

An analogue of Theorem~\ref{N-standard-thm} for Valiron deficient values is obtained in \cite{GHW}, where \eqref{Wittich-assumption} is assumed to hold as $r\to\infty$ without an exceptional set. An exceptional set is not allowed here, because the quantity $\delta_V$ is defined in terms of limit superior. This  is in contrast to the situation with the quantities $\delta_N, \delta_P, \delta_E$, which are all defined in terms of limit inferior, for which the exceptional sets in the reasoning are irrelevant.  
For these reasons, the study of Valiron standard solutions $f$ of \eqref{lden} defined by $\delta_V(a,f)=0$ for every $a\in\C\setminus\{0\}$ may not be of further interest.

%
%

\bigskip
E-mail: \texttt{janne.heittokangas@uef.fi} and \texttt{samu.pulkkinen@uef.fi}

\textsc{University of Eastern Finland, Department of Physics and Mathematics, P.O.~Box 111, 80100 Joensuu, Finland}

\medskip
E-mail: \texttt{huiy@dlmu.edu.cn}

\textsc{School of Science, Dalian Maritime University, Dalian, Liaoning, 116026, People's Republic of China}
 
\medskip
E-mail: \texttt{zmrn.amine@gmail.com}

\textsc{National Higher School of Mathematics
P.O.~Box 75, Mahelma 16093, Sidi Abdellah (Algiers), Algeria}

\end{document}